\documentclass[10pt,a4paper]{amsart}

\usepackage[utf8]{inputenc}

\usepackage{amsthm}
\usepackage{amssymb}
\usepackage{mathtools}
\usepackage{a4wide}
\usepackage{mathrsfs}
\usepackage{amsfonts}
\usepackage{enumitem}
\usepackage{hyperref}
\usepackage{dsfont}
\usepackage[dvipsnames]{xcolor}
\usepackage{orcidlink}

\newcommand{\Bcal}{\mathcal{B}}

\newcommand{\Ecal}{\mathcal{E}}
\newcommand{\Fcal}{\mathcal{F}}
\newcommand{\Gcal}{\mathcal{G}}
\newcommand{\Hcal}{\mathcal{H}}

\newcommand{\Kcal}{\mathcal{K}}

\newcommand{\Rcal}{\mathcal{R}}

\newcommand{\Ebb}{\mathbb{E}}

\newcommand{\Pbb}{\mathbb{P}}

\newcommand{\Rbb}{\mathbb{R}}

\renewcommand{\d}[1]{\, \mathrm{d} #1}

\newcommand{\FctT}{\lbrace \mathcal{F}_t \rbrace_{t \in [0, T]}}
\newcommand{\per}{\rm{per}}
\newcommand{\Id}{\operatorname{Id}}

\newtheorem{theorem}{Theorem}[section]
\newtheorem{lemma}[theorem]{Lemma}

\newtheorem{definition}[theorem]{Definition}

\theoremstyle{definition}
\newtheorem{example}[theorem]{Example}
\newtheorem{remark}[theorem]{Remark}

\numberwithin{equation}{section}

\begin{document}

\title{Absolute continuity of Rosenblatt measures}

\author{Petr \v{C}oupek\orcidlink{0000-0002-5360-6095}}
\address{Charles University, Faculty of Mathematics and Physics, Sokolovsk\'{a} 83, 186 75, Prague 8, Czech Republic
}
\email{coupek@karlin.mff.cuni.cz}

\author{Tyrone E. Duncan}
\address{University of Kansas, Department of Mathematics, Nichols Hall, 2335 Irving Hill Rd, Lawrence, KS 66045, USA.}
\email{teduncan@ku.edu}

\author{Bozenna Pasik-Duncan}
\address{University of Kansas, Department of Mathematics, Nichols Hall, 2335 Irving Hill Rd, Lawrence, KS 66045, USA.}
\email{bozenna@ku.edu}

\author{Jakub Slav\'{i}k\orcidlink{0000-0001-6465-663X}}
\address{Czech Academy of Sciences, Institute of Information Theory and Automation\protect\\Pod~Vod\'{a}renskou v\v{e}\v{z}\'{i} 4, 182 00 Prague 8, Czech Republic}
\email{slavik@utia.cas.cz}

\begin{abstract}
In the article, we address the problem of absolute continuity of translated Rosenblatt measures on the path space. In [\v{C}oupek, P., K\v{r}\'{i}\v{z}, P., Maslowski, B., \emph{Stoch.\ Proc.\ Appl.} \textbf{179} (2025) art.\ no.\ 104499], it is shown that there is no probability measure that would be equivalent to the original probability measure and under which a Rosenblatt path with a linear drift would again be a Rosenblatt path. Here, we show that if the Rosenblatt path is shifted in a direction belonging to a class of nontrivial Gaussian variables (that consists of a deterministic shift and a Wiener integral with respect to a fractional Brownian motion with a related Hurst parameter), such a measure exists. We also give several examples to demonstrate the scope of the result.
\end{abstract}



\thanks{\emph{Funding.} PČ and JS were supported by the Czech Science Foundation project no.\ 26-21423S}

\maketitle

\section{Introduction}

Absolute continuity of probability measures has been long studied. Among the first contributions is that of Cameron and Martin \cite{CamMar44}, who identified the space of shifts for which the translated Wiener measure is equivalent to the original one. This result was significantly broadened by Feldman \cite{Fel58} and H\'ajek \cite{Haj58} who established the now-classical dichotomy that two Gaussian measures are either equivalent or singular, and characterized the equivalence. In the specific case where the reference measure is the Wiener measure, Shepp \cite{She66} provided a characterization of this equivalence in terms of the mean and covariance of the second measure. This result can be interpreted as allowing (possibly noncausal) transformations relating equivalent processes. It was later realized by Hitsuda \cite{Hit68} and Duncan \cite{Dun70} that such equivalence can be realized via causal transformations.
Another generalization of Cameron and Martin's result was discovered by Girsanov \cite{Gir60}, who considered stochastic shifts. Since then, many generalizations of these result were found but only a few in genuinely non-Gaussian and non-semimartingale settings, e.g., \cite{RenYorZam08} for Dirichlet or gamma processes. In this article, we investigate absolute continuity of translated Rosenblatt processes. 

Rosenblatt processes \cite{Ros61, Taq75} form a family of stochastic processes that can be viewed as the closest non-Gaussian alternative to the family of fractional Brownian motions (abbr.\ FBM)---while both families consist of continuous centered self-similar processes with long-range memory and stationary increments, FBMs live in the first Wiener chaos of the Wiener integral with respect to the standard Wiener process and Rosenblatt processes live in the second Wiener chaos.

Rosenblatt processes arise quite naturally in a multitude of scenarios where some form of long-range dependence is present. For example, they appear in the non-central limit theorem for suitably normalized sums of strongly dependent random variables \cite{DobMaj79, LeoAhn01} or as the asymptotic distribution of the unit root statistic with errors being nonlinear transformations of long-range dependent linear processes \cite{Wu06}.

An exposition of the construction and basic properties of Rosenblatt processes can be found in the article \cite{Taq11}, a stochastic calculus for these processes is treated in \cite{CouDunDun22, Tud08}, and some of their finer properties are given in, e.g., \cite{AbrPip06, Alb98, DavKer23, DawLoo22, GarTorTud12, KerNouSakVii21, LeoPep25, Pip04}. Of course, as Rosenblatt processes belong to the more general family of Hermite processes, some additional properties shared by this class can be also found in \cite{Aya20, AyaHamLoo25, AyaTud24, BaiTaq14, CouOnd24, LooTud25, MaeTud07, MorOod86, PipTaq10}. Moreover, because Rosenblatt processes can prove useful in situations where a Gaussian model is not suitable, as in, e.g., \cite{ChaRoi20, CheEtAl19, Dom15, SkaMabSch13, SkaMabSch14, WitGotLanFraGreHei19}, they have been also considered as driving noises for stochastic (partial) differential equations (abbr.\ S(P)DE) in, e.g., \cite{BonTud11, CouMasOnd18, CouMasOnd22, SlaTud19, SlaTud19b, SlaTud19c}. Moreover, statistical inference for SDEs driven by Rosenblatt processes has also been addressed in several papers, see, e.g., \cite{AssTud20, AyaTud24, ChrTudVie09, ChrTudVie11, LooTud25, NouTra19, TudVie09}.

In particular, in the recent paper \cite{CouKriMas25}, an estimator for the drift parameter of an SDE with an additive Rosenblatt process based on a discretely observed sample path of the solution with a fixed time horizon is proposed. It is then shown that the estimator is consistent under the in-fill asymptotics, i.e.\ that the estimator is able to identify the true value of the drift parameter when the number of points at which the solution is observed increases to infinity. It is well-known that such estimation is impossible if the SDE is driven by the FBM because of its Girsanov-type theorem \cite{Bal17, DecUst99, NorValVir99}. The existence of such an estimator in the Rosenblatt case implies that the laws corresponding to different drifts are singular. It is then suggested in the article that a ``Girsanov-type theorem cannot hold true in general for the (drifted) Rosenblatt process''.

The aim of the present article is to explore the rather surprising situation. In our main result, we show that Rosenblatt process is quasi-invariant with respect to shifts that belong to a certain class of nontrivial Gaussian random variables. In particular, we show that if $(R^H_t)_{t\in [0,T]}$ is the Rosenblatt process with Hurst index $H\in (1/2,1)$ and $\theta$ is a deterministic function that belongs to the Bessel potential space $H^{\frac{H}{2},2}([0,T])$, then there exists a probability measure that is absolutely continuous with respect to the original probability measure and such that the process $(\tilde{R}_t^H)_{t\in [0,T]}$ defined by
	\begin{equation}
	\label{eq:RP_changed}
		\tilde{R}_t^H = R_t^H + 2d_H \int_0^t \theta_u\d{B}_y^{\frac{H}{2} + \frac{1}{2}} + d_H \int_0^t \theta_u^2 \d{u}, \quad t\in [0,T],
	\end{equation}
is again a Rosenblatt process with the Hurst index $H$ with respect to this new measure. Here, $B^{\frac{H}{2} + \frac12}$ is the FBM with Hurst index $\tfrac{H}{2} + \tfrac12$ and $d_H\in (0,\infty)$ is a constant given by formula \eqref{eq:dH}. This result is consistent with the result of \cite{CouKriMas25}: While processes of the form 
	\[ 
		X_t' = R_t^H + t, \quad t\in [0,T],
	\]
cannot be viewed as Rosenblatt processes under a change of measure, processes of the form 
	\[ 
		X_t'' = R_t^H + 2d_H^{-\frac{1}{2}} B_{t}^{\frac{H}{2}+ \frac{1}{2}} + t, \quad t\in [0,T],
	\]
can. Moreover, it is interesting to note that processes of the form \eqref{eq:RP_changed} have a similar structure to those that naturally appear in the It\^o formula for Rosenblatt processes in \cite{CouDunDun22}.

The article is organized as follows. In Section \ref{sect:preliminaries}, the necessary preliminaries on fractional integrals and derivatives as well as on FBMs and Rosenblatt processes are given. The main result is stated and proved in Section \ref{sect:girsanov_rp} and is followed by several examples that demonstrate its scope.

\section{Preliminaries}
\label{sect:preliminaries}

Our main tool for establishing the main result is the classical Girsanov theorem for Wiener processes, which is given in Section \ref{sect:classical_girsanov}. Section \ref{sect:fracitonal_calculus} gathers basic properties of Riemann--Liouville fractional integrals that are essential for the analysis of FBMs and Rosenblatt processes. Some basic results on FBMs and Rosenblatt processes are given in Sections \ref{sect:fbm} and \ref{sect:rp}, respectively. In the whole article, $T\in (0,\infty)$ and $(\Omega, \Fcal, \FctT, \Pbb)$ is a filtered probability space.

\subsection{Classical Girsanov theorem}
\label{sect:classical_girsanov}

Let us recall the classical Girsanov theorem for the Wiener process, see, e.g., \cite[Theorem 12.1]{Bal17}, as we will rely on it in the proof of the main theorem.

\begin{theorem}
	\label{thm:girsanov_wiener}
	Let $(B_t)_{t \in [0, T]}$ be an $\FctT$-Wiener process. Let $\phi: [0, T] \times \Omega \to \Rbb$ be an $\FctT$-progressively measurable process such that $\phi \in L^2([0, T])$ almost surely and assume that the process $(Z_t)_{t\in [0,T]}$ defined by
	\begin{equation}
		\label{eq:doleans-dane_exponential}
		Z_t = \exp\left( - \int_0^t \phi_s \d{B_s} - \frac12 \int_0^t \phi^2_s \d{s} \right), \quad t \in [0, T],
	\end{equation}
	is an $\FctT$-martingale. Let $\tilde{\Pbb}$ be the probability measure on $(\Omega, \Fcal)$ defined by $\d{\tilde{\Pbb}} = Z_T \d{\Pbb}$. Then the process $(\tilde{B}_t)_{t \in [0, T]}$ defined by
	\[
		\tilde{B}_t = B_t + \int_0^t \phi_s \d{s}, \quad t \in [0, T],
	\]
	is an $\FctT$-Wiener process with respect to $\tilde{\Pbb}$.
\end{theorem}

It is well-known that the Dol\'eans--Dade exponential $(Z_t)_{t\in [0,T]}$ defined in \eqref{eq:doleans-dane_exponential} is an $\FctT$-martingale if and only if
\[
	\Ebb Z_t = 1 \quad \text{for all} \ t \in [0, T],
\]
see e.g.\ \cite[Theorem III.5.2]{IkeWat89}. Let us also recall the Novikov condition \cite{Nov72}: If
\begin{equation}
	\label{eq:novikov}
	\Ebb \exp\left( \frac12 \int_0^T |\phi_s|^2 \d{s} \right) < \infty,
\end{equation}
then $(Z_t)_{t\in [0,T]}$ is an $\FctT$-martingale.

\subsection{Fractional calculus}
\label{sect:fracitonal_calculus}

In what follows, we recall some auxiliary results on fractional integrals and derivatives that will be needed for the analysis contained in the subsequent sections.

\begin{definition}
	Let $\alpha > 0$ and $f \in L^1([0, T])$. The \emph{left} and \emph{right  Riemann--Liouville fractional integrals of $f$ of order $\alpha$} are defined by
	\begin{align*}
		(I^\alpha_{0+} f)(x) &= \frac{1}{\Gamma(\alpha)} \int_0^x (x-y)^{\alpha-1} f(y) \d{y},
		\\
		(I^\alpha_{T-} f)(x') &= \frac{1}{\Gamma(\alpha)} \int_{x'}^{T} (y-x')^{\alpha-1} f(y) \d{y},
	\end{align*}
	respectively, for almost all $x,x' \in [0, T]$.
\end{definition}

We summarize basic properties of the Riemann--Liouville fractional integrals. For brevity, we explicitly discuss only the left fractional integral. For $\alpha > 0$, it holds that
\[
	\| I^\alpha_{0+} f \|_{L^p([0, T])} \leq \frac{T^\alpha}{\alpha \Gamma(\alpha)} \| f \|_{L^p([0, T])},
\]
see e.g.\ \cite[Eq.\ (2.72)]{SamKilMar93}. In other words, $I^\alpha_{0+}: L^p([0, T]) \to L^p([0, T])$ is a bounded operator. Moreover, it can be shown that,
for $1<p<\infty$ and $0<\alpha<1/p$, the identity $I^\alpha_{0+}(L^p([0, T])) = H^{\alpha,p}([0, T])$ holds, see e.g.\ \cite[Theorem 18.3]{SamKilMar93}, where $H^{\alpha,p}([0, T])$ denotes the Bessel potential space. By the standard embedding of the spaces $H^{\alpha, p}([0, T]) = F^\alpha_{p,2}([0, T])$ into $L^{p/(1-\alpha p)}([0, T]) = F^0_{p/(1-\alpha p), 2}([0, T])$ from e.g.\ \cite[Section 3.3.1]{Tri83}, the celebrated result by Hardy and Littlewood \cite[Theorem 4]{HarLit28}, that is the boundedness of the operator $I^\alpha_{0+}: L^p([0, T]) \to L^{p/(1-\alpha p)}([0, T])$, follows. For $1/p < \alpha < 1/p + 1$, it holds that $I^\alpha_{0+}(L^p([0, T])) = H^{\alpha,p}_0([0, T])$, where $H^{\alpha,p}_0([0, T])$ is the space of all functions $f \in H^{\alpha,p}([0, T]) \subseteq C([0, T])$ such that $f(0) = 0$.

Let $p, q \in [1, \infty)$ be such that $p^{-1}+q^{-1} \leq 1+\alpha$ with $p, q \neq 1$ if $p^{-1}+q^{-1} = 1+\alpha$. For $f \in L^p([0,T])$ and $g \in L^q([0, T])$, the fractional integration by parts
\begin{equation}
	\label{eq:integration_by_parts_fractional}
	\int_0^T f(u) (I^\alpha_{0+} g)(u) \d{u} = \int_0^T (I^\alpha_{T-} f)(u)g(u) \d{u}
\end{equation}
holds, see e.g.\ \cite[Eq.\ (2.20)]{SamKilMar93}.

\begin{definition}
	Let $\alpha \in (0, 1)$ and $f: [0, T] \to \Rbb$. The \emph{left} and \emph{right Riemann--Liouville fractional derivatives of $f$ of order $\alpha$} are defined by
	\begin{align*}
		(I^{-\alpha}_{0+} f)(x) &= \frac{1}{\Gamma(1-\alpha)} \frac{\d{}}{\d{x}} \int_0^x (x-y)^{-\alpha} f(y) \d{y},
		\\
		(I^{-\alpha}_{T-} f)(x') &= \frac{1}{\Gamma(1-\alpha)} \frac{\d{}}{\d{x'}} \int_{x'}^{T} (y-x')^{-\alpha} f(y) \d{y},
	\end{align*}
	for almost all $x,x' \in [0, T]$.
\end{definition}

The notation $I^{-\alpha}_{0+}$ and, similarly, $I^{-\alpha}_{T-}$ is justified by the fact that
\begin{equation}
	\label{eq:negative_integral}
	I^{-\alpha}_{0+} I^{\alpha}_{0+} \phi = \phi \text{ for } \phi \in L^1([0, T]) \quad \text{and} \quad I^{\alpha}_{0+} I^{-\alpha}_{0+} f = f \text{ for } f \in H^{\alpha,1}([0, T]),
\end{equation}
see e.g.\ \cite[Theorem 2.4]{SamKilMar93}.

In what follows, the operator
\[
	x^{-\alpha-\eta} (I^\alpha_{0+} (\bullet^\eta f))(x) = \frac{x^{-\alpha-\eta}}{\Gamma(\alpha)} \int_0^x (x-y)^{\alpha-1} y^\eta f(y) \d{y},
\]
where $\alpha>0$, $p \in [1, \infty]$ and $\eta > -1/p'$ will play an important role. This operator is often called the \emph{left Kober--Erd\'{e}lyi fractional integral of orders $\alpha$ and $\eta$}. By \cite[p.\ 323]{SamKilMar93}, the operator $\bullet^{-\alpha-\eta} (I^\alpha_{0+} (\bullet^\eta f): L^p([0,T]) \to L^p([0,T])$ is bounded for $1 \leq p < \infty$, see also \cite[Theorem 2]{Erd40}. In the special case $\eta=-\alpha$, the operator
\[
	(I^\alpha_{0+} (\bullet^{-\alpha} f))(x) = \frac{1}{\Gamma(\alpha)} \int_0^x (x-y)^{\alpha-1} y^{-\alpha} f(y) \d{y}
\]
with $\alpha \in (0, 1/2)$ is bounded on $L^p([0, T])$ for all $p > 1/(1-\alpha)$ and, in particular, for $p=2$ for all $\alpha \in (0, 1/2)$. We will also need the boundedness of the weighted Kober--Erd\'elyi fractional integral.

\begin{lemma}
	\label{lemma:kober-erdelyi_bounded}
	Let $\alpha \in (0, 1/2)$ and $p > 1/(1-\alpha)$.
	\begin{enumerate}
		\item The operator
		\[
			F: \phi \mapsto \bullet^{\alpha} I^{\alpha}_{0+}( \bullet^{-\alpha} \phi)
		\]
		from $L^p([0, T])$ to $H^{\alpha, p}([0, T])$ is surjective and bounded.
		\item The inverse operator $F^{-1}$ from $H^{\alpha,p}([0, T])$ to $L^p([0, T])$ is given by $f \to \bullet^{\alpha} I^{-\alpha}_{0+}(\bullet^{-\alpha} f)$.
	\end{enumerate}
\end{lemma}


\begin{proof}
	By \cite[Theorem 10.4, case 1, and Eq.\ (10.33)]{SamKilMar93}, the operator
	\[
		\prescript{}{1}I^c_{0+}(a, c) \phi = \bullet^{-a} I^c_{0+} (\bullet^{a} \phi)
	\]
	is bounded from $L^p([0, T])$ to $I^c_{0+}(L^p([0, T]))$ and is surjective for $\Re c > 0$ and $p(1-\Re a)>1$. Hence, the first claim follows by setting $a = -\alpha$ and $c=\alpha$, in other words $F = \prescript{}{1}I^\alpha_{0+}(-\alpha, \alpha)$.
	
	Regarding the second claim, we note that by \cite[Lemma 2.5]{SamKilMar93}, the operator $F$ is surjective. Thus, the inverse operator $F^{-1}: H^{\alpha,p}([0, T]) \to L^p([0, T])$ is bounded by the bounded inverse theorem. It remains to check that
	\[
		G f = \bullet^{\alpha} I^{-\alpha}_{0+}(\bullet^{-\alpha} f), \quad f \in H^{\alpha, p}([0, T]),
	\]
	is indeed the inverse operator $F^{-1}$. The identity $G F = \Id$ follows immediately from the first claim in \eqref{eq:negative_integral} and $\phi \in L^p([0, T])$. The remaining identity $F G = \Id$, follows from $\bullet^{-\alpha} \in F^{s'}_{p',q'}([0, T])$ for $0 < q' < \infty$ and $\alpha < s' - 1/p'$. Then, by \cite[Theorem 3]{Sic93},
	$\bullet^{-\alpha} f \in F^\alpha_{1,2}([0, T]) = H^{\alpha, 1}([0, T])$ and the proof is concluded by the second claim of \eqref{eq:negative_integral}.
\end{proof}

\subsection{Fractional Brownian motion}
\label{sect:fbm}

Let $H\in (1/2,1)$ and recall that the FBM with Hurst index $H$ is the centered Gaussian stochastic process with covariance function $R_H:[0,T]^2\to\Rbb$ given by 
	\begin{equation}
	\label{eq:cov_FBM} 
		R_H(s,t) = \frac{1}{2} \left( s^{2H} + t^{2H} - |t-s|^{2H}\right), \quad s,t\in [0,T].
	\end{equation}
(Of course, the definition makes sense for $H\in (0,1)$ but due to our focus on the Rosenblatt process, we restrict ourselves to the range of $H\in (1/2,1)$.) Let us recall the representation of the FBM as a Volterra-type process (i.e.\ as a stochastic convolution integral of a deterministic kernel with respect to a standard Wiener process), see, e.g., \cite[Theorem 5.2]{NorValVir99}. In particular, let $(B_t)_{t\in[0,T]}$ be a standard $\FctT$-Wiener process and denote for $t\in [0,T]$ the kernel $K_H(t,\bullet): (0,T)\to\Rbb$ by
	\[
		K_H(t,s) = c_H \mathds{1}_{[0,t]}(s) s^{\frac{1}{2} - H} \int_s^t u^{H-1/2} (u-s)^{H-\frac{3}{2}} \d{u}, \quad s\in (0,T),
	\]
	
where
\begin{equation}
	\label{eq:c_H}
	c_H = \left( \frac{H(2H-1)}{\mathrm{B}\left(2-2H, H-\tfrac{1}{2}\right)}\right)^{1/2}.
\end{equation}
Define the process $(B^H_t)_{t\in [0,T]}$ by
\begin{equation}
	\label{eq:fbm_representation}
	B^H_t = \int_0^t K_H(t,s) \d{B_s}, \quad t\in [0,T].
\end{equation}
It is straightforward to check that process $B^H$ is indeed an FBM adapted to $\FctT$.

\subsubsection{Girsanov theorem for FBM} In what follows, we recall a Girsanov-type theorem for the FBM from \cite[Theorem 4.9]{DecUst99}, see also \cite{NorValVir99}. To this end, we define the operator $K_H$ by
	\[
		(K_H f)(t) = \int_0^t K_H(t,s) f_s \d{s}, \quad t\in [0,T].
	\]
As we have the equality
	\[
		K_H f = c_H \Gamma \left( H-\tfrac12 \right) I^1_{0+}\left( \bullet^{H-\frac{1}{2}} I^{H-\frac{1}{2}}_{0+}( \bullet^{-(H-\frac{1}{2})} f) \right),
	\]
it follows by \cite[p.\ 187]{SamKilMar93} that the map
\[
	K_H: L^p([0,T]) \to I^{H+\frac{1}{2}}_{0+}(L^p([0,T]))
\]
is an isomorphism. In particular, the inverse operator
\[
	K_H^{-1}: I^{H+\frac{1}{2}}_{0+}(L^p([0,T])) \to L^p([0,T])
\]
given by
\[
	K_H^{-1} f = c_H^{-1} \Gamma \left( H-\tfrac12 \right)^{-1} \bullet^{H-\frac{1}{2}} I^{-(H-\frac{1}{2})}_{0+}\left( \bullet^{-(H-\frac{1}{2})} \frac{\d{}}{\d{x}} f \right)
\]
is well-defined, see e.g.\ \cite[Eq.\ (11)]{NuaOuk02} for details. The Girsanov theorem for the FBM from \cite[Theorem 4.9]{DecUst99} (see also \cite{NorValVir99}) can be recalled now.

\begin{theorem}
	\label{thm:girsanov_fbm}
	Let $H \in (1/2, 1)$, let $(B_t)_{t \in [0, T]}$ be an $\FctT$-Wiener process and let $(B^H_t)_{t\in [0,T]}$ be the $FBM$ defined by \eqref{eq:fbm_representation}. Let $\theta: [0, T] \times \Omega \to \Rbb$ be an $\FctT$-progressively measurable process such that 
	\begin{equation}
	\label{eq:cond_gir_fbm}
		\int_0^\bullet \theta_s \d{s} \in I^{H+\frac{1}{2}}_{0+}(L^2([0, T])), \quad \Pbb\text{-a.s.}
	\end{equation}
	Assume that the process $(Z_t)_{t\in [0,T]}$ defined by
	\[
		Z_t = \exp\left[ - \int_0^t K_H^{-1} \left( \int_0^\bullet \theta_r \d{r} \right) (s) \d{B_s} - \frac12 \int_0^t K_H^{-1} \left( \int_0^\bullet \theta_r \d{r} \right)^2(s) \d{s} \right]
	\]
	is an $\FctT$-martingale and let $\tilde{\Pbb}$ be the probability measure on $(\Omega, \Fcal)$ defined by $\d{\tilde{\Pbb}} = Z_T \d{\Pbb}$. Then the process $(\tilde{B}^H_t)_{t \in [0, T]}$ defined by
	\begin{equation}
    \label{eq:fbm_shifted}
		\tilde{B}^H_t = B^H_t + \int_0^t	\theta_s \d{s}, \quad t \in [0, T],
	\end{equation}
	is an FBM with Hurst index $H$ with respect to $\tilde{\Pbb}$ that is adapted to $\FctT$.
\end{theorem}

\begin{proof}
	It holds that
	\begin{align*}
		\int_0^t \theta_s \d{s} &= \left[ K_H \left( K_H^{-1} \left( \int_0^\bullet \theta_r \d{r} \right) \right) \right] (t)
		\\
		&= \int_0^t K_H(t,s) \d{\left( \int_0^s K_H^{-1}\left( \int_0^{\bullet} \theta_r \d{r} \right)(u) \d{u} \right)}.
	\end{align*}
	Hence, defining
	\begin{equation}
		\label{eq:fbm_psi_theta_definition}
		\phi = K_H^{-1}\left( \int_0^\bullet \theta_r \d{r} \right) = c_H^{-1} \Gamma \left( H-\tfrac12 \right)^{-1} \bullet^{H-\frac{1}{2}} I^{-(H-\frac{1}{2})}_{0+}\left( \bullet^{-(H-\frac{1}{2})} \theta \right),
	\end{equation}
	it holds that
	\[
		\tilde{B}^H_t = \int_0^t K_H(t,s) \left( \d{B_s} + \phi_s \d{s} \right)
	\]
	and the claim follows from Theorem \ref{thm:girsanov_wiener}.
\end{proof}

\begin{remark}
Since $1/2 < H+ 1/2 < 3/2$ holds, we have the equality 
	\[
		I^{H+\frac{1}{2}}_{0+}(L^2([0, T])) = H^{H+\frac{1}{2}, 2}_0([0, T]),
	\] 
and we can rewrite condition \eqref{eq:cond_gir_fbm} as 
	\[
		\theta \in H^{H-\frac{1}{2},2}([0,T]).
	\]
\end{remark}

\subsubsection{Wiener integral with respect to FBM}

Now we briefly recall the definition of the Wiener integral with respect to an FBM. This will be important because the Gaussian shifts in the Girsanov theorem for Rosenblatt processes will be described precisely as certain integrals with respect to an FBM. More details can be found in the monograph \cite{Nua06} and the references therein.

Let $\Ecal([0,T])$ be the set of real-valued step functions defined on the interval $[0,T]$ and consider the operator $\partial_1K_{H,T}^\ast: \Ecal([0,T]) \to L^2([0, T])$ defined by
	\[ 
		\partial_1K_{H,T}^\ast f = \int_\bullet^T (\partial_1K_H)(t,\bullet)f_t \d{t} = c_H \Gamma\left( H - \tfrac12 \right) \bullet^{-(H-\frac{1}{2})} I^{H-\frac{1}{2}}_{T-}(\bullet^{H-\frac{1}{2}} f).
	\] 
For $f \in \Ecal([0,T])$, the integral of $f$ with respect to the FBM $(B^H_t)_{t \in [0, T]}$ is then defined by
	\[ 
		\Ecal([0,T])\ni f = \sum_{i=1}^n F_i \mathds{1}_{[t_{i-1},t_i)} \quad\longmapsto \quad \sum_{i=1}^n F_i (B^H_{t_i}-B^H_{t_{i-1}}) = \int_0^T f_t \d{B^H}_t \in L^2(\Omega).
	\]
It can be shown that the integral is a linear isometry between $\Ecal([0,T])$ endowed with the norm $\|f\|_{\Hcal([0,T])}= \|\partial_1K_{H,T}^\ast f\|_{L^2([0,T])}$ and a closed linear subspace of $L^2(\Omega)$ endowed with the usual $L^2(\Omega)$-norm. As such, it can be uniquely extended to the completion $\Hcal([0,T])$ of $\Ecal([0,T])$ with respect to the norm $\|\cdot\|_{\Hcal([0,T])}$. Moreover, it holds for $f\in \mathcal{H}([0,T])$ that
	\begin{equation}
	\label{eq:stoch_integral_fbm_definition}
		\int_0^T f_t \d{B^H_t} = \int_0^T (\partial_1K_{H,T}^\ast f)(t) \d{B_t}.
	\end{equation}
Note that the space $\Hcal([0,T])$ was characterized in, e.g., \cite[Proposition 2.6]{CouMasOnd22}.

\subsection{Rosenblatt process}
\label{sect:rp}

Let us recall the definition of the Rosenblatt process. Let $(B_t)_{t\in [0,T]}$ be a standard $\FctT$-Wiener process and let $H \in (1/2, 1)$. For $t \in [0, T]$, define $\Kcal^H_t: (0, T)^2 \to \Rbb$ by
\begin{align*}
	\Kcal^H_t(y_1, y_2) &= d_H \mathds{1}_{[0, t]^2}(y_1,y_2) \int_{y_1 \vee y_2}^t (\partial_1K_{\frac{H}{2}+\frac{1}{2}})(u, y_1) (\partial_1K_{\frac{H}{2}+\frac{1}{2}})(u, y_2) \d{u}
	\\
	&= e_H \mathds{1}_{[0, t]^2}(y_1,y_2) (y_1 y_2)^{-\frac{H}{2}} \int_{y_1 \vee y_2}^t u^{H} (u-y_1)^{\frac{H}{2}-1} (u-y_2)^{\frac{H}{2}-1} \d{u}
\end{align*}
where
\begin{equation}
\label{eq:dH}
	d_H = \frac{1}{H+1} \left( \frac{H}{2(2H-1)} \right)^{-1/2}, \quad e_H = c_{\frac{H}{2} + \frac{1}{2}}^2 d_H,
\end{equation}
and where constant $c_{\frac{H}{2}+\frac12}$ is defined in \eqref{eq:c_H}. By using the Fubini theorem and the equality
\begin{equation}
\label{eq:int_beta}
	\frac{(uv)^{\alpha}}{\mathrm{B}\left( 1-2\alpha, \alpha \right)} \int_0^{u \wedge v} y^{-2\alpha} (u-y)^{\alpha-1} (v-y)^{\alpha-1} \d{y} = |u-v|^{2\alpha-1}, \quad u\neq v,
\end{equation}
that holds for $\alpha \in (0, 1/2)$, one can check that $\Kcal^H_t \in L^2([0, T]^2)$ for all $t \in [0, T]$. The Rosenblatt process $(R^H_t)_{t\in [0,T]}$ with Hurst parameter $H$ can be defined by
\begin{equation}
	\label{eq:rosenblatt_representation}
	R^H_t = 2 \int_0^t \int_0^{y_2} \Kcal^{H}_t (y_1, y_2) \d{B_{y_1}} \d{B_{y_2}}, \quad t\in [0,T],
\end{equation}
see, e.g., \cite[Proposition 1]{Tud08}.

\section{Absolute continuity of Rosenblatt processes}
\label{sect:girsanov_rp}

The main result of the article is given now. Examples are provided below the proof.

\begin{theorem}
	\label{thm:girsanov_rosenblatt}
	Let $H \in (1/2, 1)$ and let $(B_t)_{t \in [0, T]}$ be a standard $\FctT$-Wiener process. Let $(B^{\frac{H}{2} + \frac12}_t)_{t\in [0,T]}$ be the FBM defined by \eqref{eq:fbm_representation} (with $\frac{H}{2} + \frac12$ in place of $H$) and let $(R^H_t)_{t\in [0,T]}$ be the Rosenblatt process defined by \eqref{eq:rosenblatt_representation}. Let also 
	\[
		\theta\in H^{\frac{H}{2},2}([0,T])
	\]
be given and define the process $(Z_t)_{t\in [0,T]}$ by
	\begin{multline*}
		Z_t = \exp \left(
			\vphantom{- \frac12 \frac{1}{c_{\frac{H}{2} + \frac{1}{2}}^2 \Gamma \left( \tfrac{H}{2} \right)^2} \int_0^t s^{H} \left[ I^{-\frac{H}{2}}_{0+}(\bullet^{-\frac{H}{2}} \theta)(s) \right]^2 \d{s}}
			- \frac{1}{c_{\frac{H}{2} + \frac{1}{2}} \Gamma \left( \tfrac{H}{2} \right)} \int_0^t s^{\frac{H}{2}} I^{-\frac{H}{2}}_{0+}(\bullet^{-\frac{H}{2}} \theta)(s) \d{B_s} \right.
		\\
		\left. - \frac12 \frac{1}{c_{\frac{H}{2} + \frac{1}{2}}^2 \Gamma \left( \tfrac{H}{2} \right)^2} \int_0^t s^{H} \left[ I^{-\frac{H}{2}}_{0+}(\bullet^{-\frac{H}{2}} \theta)(s) \right]^2 \d{s} \right).
	\end{multline*}
Then the measure $\tilde{\Pbb}$ defined by $\d{\tilde{\Pbb}} = Z_T \d{\Pbb}$ is a probability measure on $(\Omega,\Fcal)$ and the process $(\tilde{R}_t^H)_{t\in [0,T]}$ defined by
	\begin{equation}
		\label{eq:girsanov_rosenblatt_shifted}
		\tilde{R}^H_t = R_t^H + 2 d_H \int_0^t \theta_u \d{B^{\frac{H}{2} + \frac{1}{2}}_u} + d_H \int_0^t \theta^2_u \d{u}
	\end{equation}
is a Rosenblatt process with Hurst index $H$ with respect to $\tilde{\Pbb}$ that is adapted to $\FctT$.
\end{theorem}

\begin{proof}
	Similarly as in \eqref{eq:fbm_psi_theta_definition}, let
	\[
		\phi = K^{-1}_{\frac{H}{2}+ \frac{1}{2}} \int_0^\bullet \theta_r \d{r} = c_{\frac{H}{2}+ \frac{1}{2}}^{-1} \Gamma \left( \tfrac{H}{2} \right)^{-1} \bullet^{\frac{H}{2}} I^{-\frac{H}{2}}_{0+}(\bullet^{-\frac{H}{2}} \theta).
	\]
	By Lemma \ref{lemma:kober-erdelyi_bounded}, it holds that $\phi \in L^2([0, T])$ and
	\begin{equation}
		\label{eq:girsanov_rosenblatt_theta_identity}
		\theta = c_{\frac{H}{2}+ \frac{1}{2}} \Gamma \left( \tfrac{H}{2} \right) \bullet^{\frac{H}{2}} I^{\frac{H}{2}}_{0+} (\bullet^{-\frac{H}{2}} \phi).
	\end{equation}
	Let $(\tilde{B}_t)_{t\in [0,T]}$ be defined by $\tilde{B}_t = B_t + \int_0^t \phi_s \d{s}$, $t\in [0,T]$. Using Novikov's condition \eqref{eq:novikov}, it follows that the process $(Z_t)_{t \in [0, T]}$ is an $\FctT$-martingale and, by Theorem \ref{thm:girsanov_wiener}, $(\tilde{B}_t)_{t \in [0, T]}$ is an $\FctT$-Wiener process on $(\Omega, \Fcal, \tilde{\Pbb})$. The claim will follow once we establish the equality
	\begin{equation}
	\label{eq:rosenblatt_girsanov_product_formula}
		\begin{aligned}
			\tilde{R}^H_t &= 2 \int_0^t \int_0^{y_2} \Kcal^{H}_t(y_1,y_2) \d{\tilde{B}_{y_1}} \d{\tilde{B}_{y_2}}
			\\
			&= 2 \int_0^t \int_0^{y_2} \Kcal^{H}_t (y_1,y_2) \left( \d{B_{y_1}} + \phi_{y_1} \d{y_1} \right) \left( \d{B_{y_2}} + \phi_{y_2} \d{y_2} \right).
		\end{aligned}
	\end{equation} 
	
	Let us start with the deterministic term. By \eqref{eq:girsanov_rosenblatt_theta_identity}, the Fubini theorem and the symmetry of the kernel $\Kcal^H_t$, it holds that
	\begin{equation}
		\label{eq:rosenblatt_girsanov_deterministic_term}
		\begin{aligned}
			d_H \int_0^t \theta_u^2 \d{u} &= e_H \Gamma \left( \tfrac{H}{2} \right)^2 \int_0^t \left[ u^{\frac{H}{2}} I^{\frac{H}{2}}_{0+}( \bullet^{-\frac{H}{2}} \phi)(u) \right]^2 \d{u}
			\\
			&= e_H \int_0^t \left[ u^{\frac{H}{2}} \int_0^u (u-y)^{\frac{H}{2}-1} y^{-\frac{H}{2}} \phi_y \d{y} \right]^2 \d{u}
			\\
			&= \int_0^t \int_0^t \Kcal^H_t(y_1, y_2) \phi_{y_1} \phi_{y_2} \d{y_1} \d{y_2}
			\\
			&= 2 \int_0^t \int_0^{y_2} \Kcal^H_t(y_1, y_2) \phi_{y_1} \phi_{y_2} \d{y_1} \d{y_2}.
		\end{aligned}
	\end{equation}
	
	We continue to the stochastic integral. Recalling the definition of the stochastic integral with respect to the FBM \eqref{eq:stoch_integral_fbm_definition} and \eqref{eq:girsanov_rosenblatt_theta_identity}, we deduce
	\begin{align*}
		&2 d_H \int_0^t \theta_{y} \d{B^{\frac{H}{2}+ \frac{1}{2}}_{y}}
		\\
		&\quad = \frac{2 e_H}{c_{\frac{H}{2}+ \frac{1}{2}}} \int_0^t y_2^{-\frac{H}{2}} \left[ \int_{y_2}^t u^H (u-y_2)^{\frac{H}{2}-1} u^{-\frac{H}{2}} \theta_u \d{u} \right] \d{B_{y_2}}
		\\
		&\quad = 2 e_H \int_0^t y_2^{-\frac{H}{2}} \left[ \int_{y_2}^t u^H (u-y_2)^{\frac{H}{2}-1} \right.
		\\
		&\quad \hphantom{= 2 e_H \int_0^t y_2^{-\frac{H}{2}} \bigg[ \ } \times \left. \left( \int_0^u (u-y_1)^{\frac{H}{2}-1} y_1^{-\frac{H}{2}} \phi_{y_1} \d{y_1} \right) \d{u} \right] \d{B_{y_2}}
		\\
		&\quad = \Xi_t + \Psi_t,
	\end{align*}
	where
	\[
		\Xi_t = 2 e_H \int_0^t y_2^{-\frac{H}{2}} \left[ \int_{y_2}^t u^H (u-y_2)^{\frac{H}{2}-1} \right. \left. \left( \int_0^{y_2} (u-y_1)^{\frac{H}{2}-1} y_1^{-\frac{H}{2}} \phi_{y_1} \d{y_1} \right) \d{u} \right] \d{B_{y_2}}
	\]
	and
	\[
		\Psi_t = 2 e_H \int_0^t y_2^{-\frac{H}{2}} \left[ \int_{y_2}^t u^H (u-y_2)^{\frac{H}{2}-1} \right.
		 \left. \left( \int_{y_2}^u (u-y_1)^{\frac{H}{2}-1} y_1^{-\frac{H}{2}} \phi_{y_1} \d{y_1} \right) \d{u} \right] \d{B_{y_2}}.
	\]
	We employ \eqref{eq:girsanov_rosenblatt_theta_identity} and the Fubini theorem to obtain
	\begin{equation}
	\label{eq:rosenblatt_girsanov_det_stoch_term}
		\Xi_t = 2 \int_0^t \int_0^{y_2} \Kcal^H_t(y_1, y_2) \phi_{y_1} \d{y_1} \d{B_{y_2}}.
	\end{equation}
	
	It remains to deal with the term $\Psi_t$. A straightforward application of the stochastic Fubini theorem is not possible due to the lack of integrability and hence, we resort to a variant of the semimartingale approximation from \cite{Tud08}. Let us denote
	\[
		\Psi^\varepsilon_t = 2 e_H \int_0^t y_1^{-\frac{H}{2}} \left[ \int_{y_1}^t u^H (u+\varepsilon-y_1)^{\frac{H}{2}-1} \right.
		 \left. \left( \int_{y_1}^u (u-y_2)^{\frac{H}{2}-1} y_2^{-\frac{H}{2}} \phi_{y_2} \d{y_2} \right) \d{u} \right] \d{B_{y_1}}.
	\]
	Note that, compared to $\Psi$, we have switched $y_1$ and $y_2$. We aim to establish the convergence $\Psi^\varepsilon_t \xrightarrow[\varepsilon\to 0]{L^2(\Omega)} \Psi_t$. Using the It\^o isometry, we deduce
	\[
		\| \Psi^\varepsilon_t - \Psi_t \|_{L^2(\Omega)}^2 \lesssim \int_0^t y_1^{-H} g_\varepsilon(y_1)^2 \d{y_1},
	\]
	where
	\[
		g_\varepsilon(y_1) = \int_{y_1}^t u^{H} \left[ (u-y_1)^{\frac{H}{2}-1} - (u+\varepsilon-y_1)^{\frac{H}{2}-1} \right]
		 \left| \int_{y_1}^u (u-y_2)^{\frac{H}{2}-1} y_2^{-\frac{H}{2}} \phi_{y_2} \d{y_2} \right| \d{u}.
	\]
	We directly estimate
	\begin{align*}
		g_\varepsilon(y_1) &\lesssim \int_{y_1}^t u^H (u-y_1)^{\frac{H}{2}-1} \left( \int_{y_1}^u (u-y_2)^{\frac{H}{2}-1} y_2^{\frac{H}{2}} |\phi_{y_2}| \d{y_2} \right) \d{u}
		\\
		&\lesssim I^{\frac{H}{2}}_{t-}\left( \bullet^H I^{\frac{H}{2}}_{0+}(\bullet^{-\frac{H}{2}} |\phi|) \right) (y_1).
	\end{align*}
	Clearly, $(\bullet-y_1)^{H/2-1} \in L^p([y_1, t])$ for all $p<(1-\tfrac{H}{2})^{-1}$. Since $|\phi| \in L^2([0, t])$, Lemma \ref{lemma:kober-erdelyi_bounded} yields $\bullet^{H/2} I^{H/2}_{0+}(\bullet^{-H/2} |\phi|) \in H^{H/2,2}([0, t]) \subseteq L^{2/(1-H)}([0, t])$ and hence $\bullet^{H} I^{H/2}_{0+}(\bullet^{-H/2} |\phi|) \in L^{2/(1-H)}([0, t])$. Thus, by the dominated convergence theorem and the standing assumption $H \in (1/2, 1)$,
	$g_\varepsilon \xrightarrow[\varepsilon\to 0]{} 0$ almost everywhere on $[0, t]$. The properties of the Riemann--Liouville fractional integral in Section \ref{sect:fracitonal_calculus} and standard embedding theorem \cite[Section 3.3.1]{Tri83} imply that $I^{H/2}_{t-}(\bullet^{H} I^{H/2}_{0+}(\bullet^{-H/2} |\phi|)) \in H^{H/2,2/(1-H)}_0([0, T]) \subseteq C([0, t])$. In particular, up to a multiplicative constant, the function $\bullet^{-H}$ is an integrable dominating function and convergence $\Psi^\varepsilon_t \xrightarrow[\varepsilon\to 0]{L^2(\Omega)} \Psi_t$ follows by another use of the dominated convergence theorem. For $\varepsilon>0$, we define
	\[
		\Kcal^H_{t,\varepsilon}(y_1, y_2) = e_H \mathds{1}_{[0, t]^2}(y_1,y_2) (y_1 y_2)^{-\frac{H}{2}} \int_{y_1 \vee y_2}^t u^{H} (u+\varepsilon-y_1)^{\frac{H}{2}-1} (u-y_2)^{\frac{H}{2}-1} \d{u}.
	\]
	Employing the stochastic Fubini theorem, we observe that
	\[
		\Psi^\varepsilon_t = 2 \int_0^t \int_0^{y_2} \Kcal^H_{t, \varepsilon}(y_1, y_2) \d{B_{y_1}} \phi_{y_2} \d{y_2}.	
	\]
	Similarly as in \cite[Proposition 2]{Tud08}, we may show that
	\[
		\Psi^{\varepsilon}_t \quad\xrightarrow[\varepsilon\to 0]{L^2(\Omega)} \quad 2 \int_0^t \int_0^{y_2} \Kcal^H_{t}(y_1, y_2) \d{B_{y_1}} \phi_{y_2} \d{y_2}
	\]
	and therefore,
	\begin{equation}
		\label{eq:rosenblatt_girsanov_stoch_det_term}
		\Psi_t = 2 \int_0^t \int_0^{y_2} \Kcal^H_{t}(y_1, y_2) \d{B_{y_1}} \phi_{y_2} \d{y_2}.
	\end{equation}
	Collecting equalities \eqref{eq:rosenblatt_girsanov_deterministic_term}, \eqref{eq:rosenblatt_girsanov_det_stoch_term} and \eqref{eq:rosenblatt_girsanov_stoch_det_term}, we verify that the representation in \eqref{eq:rosenblatt_girsanov_product_formula} holds and this concludes the proof.
\end{proof}

In the following remark, we express $R^H$ in equation \eqref{eq:RP_changed} in terms of processes $\tilde{R}^H$ and $\tilde{B}^{\frac{H}{2}+\frac12}$.

\begin{remark}
\label{rem:R_wrt_new_B}
	Let $\phi$ and $(\tilde{B}_t)_{t \in [0, T]}$ be as in the proof of the above theorem. By Theorem \ref{thm:girsanov_fbm}, the process $(\tilde{B}^{\frac{H}{2}+\frac12}_t)_{t \in [0, T]}$ defined by $\tilde{B}^{\frac{H}{2}+\frac12}_t = B^{\frac{H}{2}+\frac12}_t + \int_0^t \theta_s \d{s}$, $t\in [0,T]$, is an FBM with Hurst parameter $\frac{H}{2}+\frac12$ with respect to the measure $\tilde{\Pbb}$. By using the fractional integration by parts \eqref{eq:integration_by_parts_fractional} and the second relation in \eqref{eq:negative_integral}, it holds that
	\begin{align*}
		&2 d_H \int_0^t \theta_{u} \d{B^{\frac{H}{2} + \frac{1}{2}}_{u}}
		\\
		&\quad = 2d_H c_{\frac{H}{2} + \frac{1}{2}} \Gamma \left( \tfrac{H}{2} \right) \int_0^t u^{-\frac{H}{2}} I^{\frac{H}{2}}_{t-}(\bullet^{\frac{H}{2}} \theta)(u) \d{\tilde{B}_{u}}
		\\
		&\quad \hphantom{= } \quad - 2d_H c_{\frac{H}{2} + \frac{1}{2}} \Gamma \left( \tfrac{H}{2} \right) \int_0^t u^{-\frac{H}{2}} \phi_u I^{\frac{H}{2}}_{t-}(\bullet^{\frac{H}{2}} \theta)(u) \d{u}
		\\
		&\quad = 2 d_H \int_0^t \theta_{u} \d{\tilde{B}^{\frac{H}{2} + \frac{1}{2}}_{u}} - 2d_H c_{\frac{H}{2} + \frac{1}{2}} \Gamma \left( \tfrac{H}{2} \right) \int_0^t u^{\frac{H}{2}} \theta_y I^{\frac{H}{2}}_{0+}(\bullet^{-\frac{H}{2}} \phi)(u) \d{u}
		\\
		&\quad = 2 d_H \int_0^t \theta_{u} \d{\tilde{B}^{\frac{H}{2} + \frac{1}{2}}_{u}} - 2d_H \int_0^t \theta_u^2 \d{u}.
	\end{align*}
	Thus, \eqref{eq:girsanov_rosenblatt_shifted} can be written as
	\[
		R^H_t = 	\tilde{R}^H_t - 2 d_H \int_0^t \theta_{u} \d{\tilde{B}^{\frac{H}{2} + \frac{1}{2}}_{u}} + d_H \int_0^t \theta_u^2 \d{u}.
	\]
\end{remark}

In the next example, we study the accessibility of the underlying Wiener process, and therefore also of the Radon--Nikodym derivative, from models involving FBM and Rosenblatt process.

\begin{example}
In what follows, we denote by $(\Fcal^X_t)_{t\in [0,T]}$ the natural filtration of a stochastic process $(X_t)_{t\in [0,T]}$ and by $\overline{\Gcal}$ the completion of a $\sigma$-algebra $\Gcal$ with respect to $\Pbb$.

\emph{Case 1: FBM.} If $B^H$ is the FBM defined by \eqref{eq:fbm_representation}, we can recover the process $B$. To this end,  recall the operator 
	\[ 
		\partial_1K_{H,T}^\ast f = c_H \Gamma\left( H - \tfrac12 \right) \bullet^{-(H-\frac{1}{2})} I^{H-\frac{1}{2}}_{T-}(\bullet^{H-\frac{1}{2}} f).
	\] 
The operator admits a right inverse that is given by 
	\[ 
		(\partial_1K_{H,T}^\ast)^{-1}\psi = c_H^{-1}\Gamma\left( H-\tfrac12\right)^{-1} \bullet^{-(H-\frac{1}{2})}I_{T-}^{-(H-\frac{1}{2})} (\bullet^{H-\frac{1}{2}} \psi)
	\]
and we have, for $t\in (0,T)$, that 
	\[ 
		\int_0^T ((\partial_1 K_{H,T}^\ast)^{-1} \mathds{1}_{(0,t)})(s)\d{B}_s^H = \int_0^T \mathds{1}_{(0,t)}(s)\d{B}_s = B_t
	\]
holds by \eqref{eq:stoch_integral_fbm_definition}. Since $((\partial_1 K_{H,T}^\ast)^{-1} \mathds{1}_{(0,t)})(s) = 0$ for $s>t$, the left-hand side of the above equality is $\mathcal{F}_t^{B^H}$-measurable and so we obtain that $\overline{\Fcal}_t^{B^H} \supseteq \overline{\Fcal}_t^B$. Of course, as $B^H$ is defined by \eqref{eq:fbm_representation}, we also immediately obtain $\mathcal{F}^{B^H}_t \subseteq \Fcal_t^B$ and therefore, we have that $\overline{\Fcal}^{B_H}_t = \overline{\Fcal}^B_t$.

\emph{Case 2: Rosenblatt process.} On the other hand, this is not the case of the Rosenblatt process. Intuitively, this is because the Rosenblatt process is an even functional of the Wiener process. Let us make the intuition more precise. 

Let $(\Omega,\Fcal,\Pbb) = (C([0,T]), \mathcal{B}(C([0,T])), \Pbb_B)$ be the classical Wiener space and let $B:\Omega\to\Omega$, $B(\omega)=\omega$, be the coordinate Wiener process. Let $R^H$ be the Rosenblatt process generated by $B$ via the representation \eqref{eq:rosenblatt_representation}. Define the map $\vartheta:\Omega\to\Omega$ by $\vartheta\omega = -\omega$. For each $t\in [0,T]$, the random variable $R_t^H$ can be viewed as the second-order multiple Wiener--It\^o integral (w.r.\ to $B$) and hence, $R_t^H(\vartheta\omega)=R_t^H(\omega)$, $\omega\in\Omega$. It follows that the map $\Rcal^H_t: \Omega\to C([0,t])$, $\Rcal^H_t(\omega)=(R_s^H(\omega))_{s\in [0,t]}$, satisfies $\Rcal_t^H(\vartheta\omega)=\Rcal_t^H(\omega)$ for every $\omega\in\Omega$. Consider also the map $\Bcal_t :\Omega\to C([0,t])$, $\Bcal_t(\omega) = (B_s(\omega))_{s\in [0,t]}$ and assume that there exists a measurable map $\varPhi: C([0,t])\to C([0,t])$ such that $\Bcal_t = \varPhi (\Rcal_t^H)$ $\Pbb_B$-almost surely. Then we have that 
	\[ 
		-\Bcal_t(\omega) = \Bcal_t(\vartheta\omega) = \varPhi(\Rcal_t^H(\vartheta\omega)) = \varPhi(\Rcal_t^H(\omega)) = \Bcal_t(\omega)
	\]
holds for $\Pbb_B$-almost every $\omega\in\Omega$ so that $\Bcal_t=0$ almost surely which is a contradiction. Hence, there is no measurable map $\varPhi: C([0,t])\to C([0,t])$ such that $\Bcal_t = \varPhi (\Rcal_t^H)$ holds almost surely, i.e.\ the underlying Wiener process cannot be recovered from the Rosenblatt process. 

\emph{Case 3: Mix of FBM and Rosenblatt process.} Let $H, H'\in (1/2,1)$ and let $B$ be a standard Wiener process. Let $B^{H}$ and $R^{H'}$ be the FBM and RP defined by \eqref{eq:fbm_representation} and \eqref{eq:rosenblatt_representation}, respectively. We stress that both the FBM and Rosenblatt process are defined via the same Wiener process. Consider process $(Z_t)_{t\in [0,T]}$ defined by $Z_t=B^{H}_t + R^{H'}_t$, $t\in [0,T]$. Then $\overline{\Fcal}^Z_t = \overline{\Fcal}_t^{B}$, $t\in [0,T]$. Indeed, we again clearly have that $\Fcal_t^Z\subseteq \Fcal_t^B$ by the construction of process $Z$. On the other hand, if we denote by $P_{\mathcal{H}_1}: L^2(\Omega)\to L^2(\Omega)$ the orthogonal projection on the first Wiener chaos $\mathcal{H}_1$ of process $B$, we have that $P_{\mathcal{H}_1}(Z_t) = B^{H}_t$, $t\in [0,T]$, and since $\overline{\Fcal}_t^{B^{H}} = \overline{\Fcal}_t^{B}$, $t\in [0,T]$, by the result in \emph{Case~1} above, we obtain $\overline{\Fcal}_t^Z \supseteq \overline{\Fcal}_t^B$, $t\in [0,T]$. 
\end{example}

In the following example, we illustrate that even simple deterministic shifts of the underlying Wiener process will shift a Rosenblatt process in a stochastic direction.

\begin{example}
	Let $\phi_u = u^{\alpha}$ for some $\alpha > -\tfrac12$. By Novikov's condition, the process $(Z_t)_{t \in [0, T]}$ defined by
	\[
		Z_t = \exp\left( - \int_0^t s^\alpha \d{B_s} - \frac{1}{2(2\alpha+1)} t^{2\alpha+1} \right), \quad t\in [0,T],
	\]
	is an $\FctT$-martingale. Let $\theta: [0, T] \to \Rbb$ be defined by \eqref{eq:girsanov_rosenblatt_theta_identity}. Using \cite[Eq.\ (2.44)]{SamKilMar93}, it holds that
	\[
		\theta_u = c_{\frac{H}{2} + \frac{1}{2}} \mathrm{B}\left(\tfrac{H}{2}, 1+\alpha-\tfrac{H}{2} \right) u^{\alpha+\frac{H}{2}}, \quad u\in [0,T],
	\]
	and hence, the process $(\tilde{R}_t^H)_{t\in [0,T]}$ defined by
	\begin{multline*}
		\tilde{R}^H_t = R^H_t + 2 d_H c_{\frac{H}{2} + \frac{1}{2}} \mathrm B\left( \tfrac{H}{2}, 1+\alpha-\tfrac{H}{2} \right) \int_0^t u^{\alpha+\frac{H}{2}} \d{B^{\frac{H}{2} + \frac{1}{2}}_u}
		\\
		+ \frac{e_H}{2\alpha+H+1} \mathrm{B}\left(\tfrac{H}{2}, 1+\alpha-\tfrac{H}{2} \right)^2 t^{2\alpha+H+1}
	\end{multline*}
	is an Rosenblatt process on $(\Omega, \Fcal, \tilde{\Pbb})$ where $\d{\tilde{\Pbb}} = Z_T \d{\Pbb}$. We observe that both choices $\alpha = 0$ and $\alpha=-\tfrac{H}{2}$ corresponding to the simplest deterministic shifts of the underlying Wiener process $(B_t)_{t \in [0, T]}$ and the Rosenblatt process $(R^H_t)_{t \in [0, T]}$, respectively, require an additional stochastic integral with respect to the FBM $(B^{\frac{H}{2}+\frac12}_t)_{t \in [0, T]}$.
\end{example}

We also illustrate that the stochastic shift is not uniquely determined by the deterministic shift.

\begin{example}
	Let $A \subseteq [0, T]$ be a measurable set of positive measure such that $\per\,A < \infty$ where the perimeter of the set $A$ is defined by
	\[
		\per \ A = \| \mathds{1}_A \|_{BV(\Rbb)}. 	
	\]	
	By, e.g., \cite[Lemma 3]{Sic21}, it holds that
	\[	
		\mathds{1}_A \in B^{\frac12}_{2,\infty}([0, T]) \subseteq B^{\frac{H}{2}}_{2,2}([0, T]) = F^{\frac{H}{2}}_{2,2}([0, T]) = H^{\frac{H}{2},2}([0, T]).
	\]
	Let $\theta^A_u = \mathds{1}_{A}(u) - \mathds{1}_{[0, T] \setminus A}(u)$, 
	then $(\theta^A)^2 = 1$ almost everywhere on $[0, T]$ and, by the above, $\theta^A \in H^{\frac{H}{2},2}([0, T])$. By Theorem \ref{thm:girsanov_rosenblatt}, the process
	\[
		\tilde{R}^{H,A}_t = R^H_t + 2 d_H \int_0^t \theta^A_y \d{B^{\frac{H}{2} + \frac{1}{2}}_y} + d_H t
	\]
	is a Rosenblatt process on $(\Omega, \Fcal, \tilde{\Pbb}_A)$ where $\d{\tilde{\Pbb}_A} = Z^{A}_T \d{\Pbb}$ for suitably defined $Z^A_T$. 	In particular, the processes $(\tilde{R}^{H,A}_t)_{t \in [0, T]}$ are Rosenblatt processes defined on probability spaces with equivalent probability measures $\tilde{\Pbb}_A$ constructed from the original Rosenblatt process $(R^H_t)_{t \in [0, T]}$ by the same deterministic shift supplemented by different Gaussian shifts.
\end{example}

Finally, we address the possibility of drift removal from models involving a Rosenblatt process.

\begin{remark}
    Let $H$, $B$, $B^{\frac{H}{2}+ \frac{1}{2}}$, $R^H$, $\theta$, and $\tilde\Pbb$ be defined as in Theorem~\ref{thm:girsanov_rosenblatt}.
   Then by Theorem~\ref{thm:girsanov_fbm} and by Theorem \ref{thm:girsanov_rosenblatt}, processes $(\tilde{B}_t^{\frac{H}{2} + \frac{1}{2}})_{t\in [0,T]}$ and $(\tilde{R}_t^H)_{t\in [0,T]}$ defined by \[\tilde{B}_t^{\frac{H}{2} + \frac{1}{2}} = B_{t}^{\frac{H}{2} + \frac{1}{2}} + \int_0^t \theta_s\d{s}\quad \mbox{and}\quad \tilde{R}_t^H = R_t^H + 2d_H\int_0^t \theta_s \d{B}_s^{\frac{H}{2} + \frac{1}{2}} + d_H \int_0^t \theta_s^2\d{s},\quad t\in [0,T],\] are an FBM with Hurst index $\frac{H}{2} + \frac{1}{2}$ and Rosenblatt process with Hurst index $H$ under the measure $\tilde{\Pbb}$, respectively. 
   
    Let $a\in L^1([0,T])$ and $b\in L^2([0,T])$, and let $(X_t)_{t\in [0,T]}$ be defined by
        \[ 
            X_t = \int_0^t a_s\d{s} + \int_0^t b_s\d{B}_s^{\frac{H}{2} + \frac{1}{2}} + R_t^H, \quad t\in [0,T].
        \]
   Similarly as in Remark \ref{rem:R_wrt_new_B}, we can rewrite $X$ in terms of $\tilde{B}^{\frac{H}{2} + \frac{1}{2}}$ and $\tilde{R}^H$ as follows:
    \begin{equation}
    \label{eq:X}
        X_t = \int_0^t \left(a_s - b_s\theta_s + d_H\theta_s^2\right)\d{s} + \int_0^t \left(b_s - 2d_H \theta_s\right)\d\tilde{B}_s^{\frac{H}{2} + \frac{1}{2}} + \tilde{R}_t^H, \quad t\in [0,T].
    \end{equation}
It follows that if $ D_s=b_s^2 - 4d_Ha_s \geq 0$, $s\in [0,T]$, then one can choose $\theta_s = (2d_H)^{-1}(b_s \pm D_s^\frac{1}{2})$, $s\in [0,T]$, provided of course that such $\theta$ belongs to $ H^{\frac{H}{2},2}([0,T])$, so that 
    \[
        X_t = \pm\int_0^t D_s^\frac{1}{2}\d\tilde{B}_s^{\frac{H}{2} + \frac{1}{2}} + \tilde{R}_t^H, \quad t\in [0,T].
    \]
If, moreover, $D_s=0$, $s\in [0,T]$, we obtain
    \[
        X_t = \tilde{R}_t^H, \quad t\in [0,T].
    \]
In other words, models of the form \eqref{eq:X} need to have a specific structure ($D_s\geq 0$) should Theorem~\ref{thm:girsanov_rosenblatt} be used to remove the deterministic drift. If $D>0$, such change of measure would yield a centered process. If $D=0$, then this change of measure would yield not only a centered process but a Rosenblatt process with Hurst index $H$.
\end{remark}

\begin{remark}
We expect that by the same technique given here a similar result can be established for the more general class of Hermite processes or even for  broader classes of stochastic processes that admit a representation as multiple Wiener--It\^o integral with a deterministic kernel. 
\end{remark}

\end{document}